\documentclass[12pt,etds]{amsart}
\usepackage{amsmath}
\usepackage{amssymb}
\usepackage{amsfonts}
\usepackage{amsthm}
\usepackage{enumerate}
\usepackage{tikz}
\usetikzlibrary{matrix}

\usepackage{pb-diagram}

\newtheorem{theorem}{Theorem}[section]
\newtheorem{lemma}[theorem]{Lemma}
\newtheorem{prop}[theorem]{Proposition}

\theoremstyle{remark}

\newtheorem{remark}[theorem]{Remark}

\def\N{{\mathbb N}}
\def\O{{\mathcal{O}}}

\def\TT{{\mathbb T}}

\def\Z{{\mathbb Z}}

\def\A{{\mathcal{A}}}
\def\B{{\mathcal{B}}}

\def\K{{\mathcal{K}}}

\def\I{{\mathcal{I}}}
\def\J{{\mathcal{J}}}

\newcommand{\clsp}{\overline{\operatorname{span}}}

\newcommand{\lt}{\operatorname{lt}}
\newcommand{\id}{\operatorname{id}}

\newcommand{\piso}{\operatorname{piso}}

\newcommand{\Ind}{\operatorname{Ind}}
\newcommand{\ind}{\operatorname{\textsf{ind}}}

\newcommand{\Prim}{\operatorname{Prim}}

\newcommand{\whitesquare}{\hfill $\whitesquare$\newline\vspace{0.4cm}}

\numberwithin{equation}{section}

\begin{document}

\title[Primitive ideal space of the semigroup crossed product]
{The primitive ideal space of the partial-isometric crossed product by automorphic actions of the semigroup $\N^{2}$}

\author[Saeid Zahmatkesh]{Saeid Zahmatkesh}
\address{Mathematics and Statistics with Applications (MaSA), Department of Mathematics, Faculty of Science, King Mongkut's University
of Technology Thonburi, Bangkok 10140, THAILAND}
\email{saeid.zk09@gmail.com, saeid.kom@kmutt.ac.th}




\subjclass[2010]{Primary 46L55}
\keywords{$C^*$-algebra, automorphism, partial isometry, crossed product, primitive ideal}

\begin{abstract}
Let $(A,\mathbb{N}^{2},\alpha)$ be a dynamical system consisting of a $C^*$-algebra $A$ and an action $\alpha$ of $\mathbb{N}^{2}$ on $A$ by
automorphisms. Let $A\times_{\alpha}^{\piso}\mathbb{N}^{2}$ be the partial-isometric crossed product of the system. We apply the fact that
it is a full corner of a crossed product by the group $\Z^{2}$ in order to give a complete description of its primitive ideal space.
\end{abstract}
\maketitle

\section{Introduction}
\label{intro}
It is known that if a $C^*$-algebra $\B$ is a full corner in a $C^*$-algebra $\A$, then they are Morita equivalent, and hence, their
primitive ideal spaces are homeomorphic (see \cite{RW}). Recall that the primitive ideal space of a $C^*$-algebra $\A$ is denoted by
$\Prim \A$. Now, in the present work, we consider the dynamical system $(A,\N^{2},\alpha)$
in which $\N^{2}$ denotes the positive cone of the (abelian lattice-ordered) group $\Z^{2}$, $A$ is a $C^*$-algebra, which is not
necessarily unital, and $\alpha$ is an action of $\mathbb{N}^{2}$ on $A$ by automorphisms such that $\alpha_{0}=\id$. Our goal is
to describe the primitive ideal space of the partial-isometric crossed product $A\times_{\alpha}^{\piso}\N^{2}$ of the system and its
hull-kernel (Jacobson) topology completely. To do so, since $A\times_{\alpha}^{\piso}\N^{2}$ is a full corner in a group crossed product $(B_{\Z^{2}}\otimes A)\rtimes\Z^{2}$ (see \cite[\S 5]{SZ2}), it suffices to describe $\Prim((B_{\Z^{2}}\otimes A)\rtimes\Z^{2})$, for
which, we then apply the works on the ideal structure of crossed products by groups available in \cite{ECH,W}. So, we need to consider
the following two conditions:
\begin{itemize}
\item[(i)] when $A$ is separable and abelian;
\item[(ii)] when $A$ is separable and $\Z^{2}$ acts on $\Prim A$ freely. 
\end{itemize}
Under the first condition, we apply \cite[Theorem 8.39]{W} to see that $\Prim((B_{\Z^{2}}\otimes A)\rtimes\Z^{2})$ is
homeomorphic to a quotient of the product space
\begin{align}
\label{PS}
\Delta(B_{\Z^{2}})\times \Delta(A)\times\widehat{\Z^{2}}=\Delta(B_{\Z})\times\Delta(B_{\Z})\times \Delta(A)\times \TT^{2},
\end{align}
where $\Delta(B_{\Z})$ and $\Delta(A)$ are the spectrums of the (abelian) $C^*$-algebras $B_{\Z}$ and $A$, respectively.
Then, the quotient space is identified by the disjoint union
\begin{align}
\label{DISJ-U}
\Delta(A)\sqcup \Prim(A\rtimes_{\dot{\alpha}}\Z)\sqcup \Prim(A\rtimes_{\ddot{\alpha}}\Z)\sqcup \Prim (A\rtimes_{\alpha}\Z^{2})
\end{align}
through parameterizing the equivalent classes, where $\dot{\alpha}$ and $\ddot{\alpha}$ are two automorphisms corresponding to two
generators of the group $\Z^{2}$. Finally, the open sets in (\ref{DISJ-U}) are precisely identified by using the fact the quotient map
of (\ref{PS}) onto (\ref{DISJ-U}) is open (see \cite[Remark 8.40]{W}).
Under the second condition, we apply \cite[Corollary 7.35]{ECH} to see that $\Prim((B_{\Z^{2}}\otimes A)\rtimes\Z^{2})$ is homeomorphic
to a quotient of the product space
$$\Prim (B_{\Z^{2}}\otimes A)=\Prim B_{\Z^{2}} \times \Prim A=\Prim B_{\Z} \times \Prim B_{\Z} \times \Prim A.$$
This quotient space is called the \emph{quasi-orbit space}, which, by a similar discussion to the first condition, will be
described along with its quotient topology precisely. Note that \cite[Corollary 3.13]{SZ4} already indicates that the primitive ideals
of $A\times_{\alpha}^{\piso} \N^{2}$ are coming from the four sets
$$\Prim A, \Prim(A\rtimes_{\dot{\alpha}}\Z), \Prim(A\rtimes_{\ddot{\alpha}}\Z),\ \textrm{and}\ \Prim (A\rtimes_{\alpha}\Z^{2}).$$
So, all these ideals will also be identified in the present work under the conditions (i) and (ii) mentioned earlier.
We would like to mention that the present work is therefore a generalization of the effort in \cite{LZ} based on the results of
\cite{AZ}. To study the theory of the partial-isometric crossed products, readers may refer to \cite{Fowler,LR,SZ3} as a preliminary background.
Further studies in this regard are available in \cite{Adji-Abbas,AZ,AZ2,LZ,SZ,SZ2,SZ4}.

We begin with a preliminary section containing a quick recall on some results from \cite{SZ3,SZ4,SZ2}, and a brief discussion
on the primitive ideal space of crossed products by groups from \cite{W,ECH}. In section \ref{prim A}, we identify the primitive ideals
of the algebra $A\times_{\alpha}^{\piso}\N^{2}$ derived from $\Prim A$.
In sections \ref{case 1} and \ref{case 2}, by applying the realization of $A\times_{\alpha}^{\piso}\N^{2}$
as a full corner of a crossed product by the group $\Z^{2}$, $\Prim(A\times_{\alpha}^{\piso}\N^{2})$ is completely described under
some certain conditions. Moreover, we identify all the primitive ideals of $A\times_{\alpha}^{\piso}\N^{2}$, and provide necessary
and sufficient conditions under which $A\times_{\alpha}^{\piso}\N^{2}$ is GCR (type I or postliminal). In section \ref{sec:last}, the
final section, we see that $A\times_{\alpha}^{\piso}\N^{2}$ is primitive precisely when $A$ is primitive.
\section{Preliminaries}
\label{sec:pre}

\subsection{The algebra $A\times_{\alpha}^{\piso}\N^{2}$ as a full corner}
\label{piso fc}
First of all, since $\Z^{2}$ is an additive group, the notation ``$+$" is used for its action in the present work.

Suppose that $(A,\mathbb{N}^{2},\alpha)$ is a dynamical system consisting of a $C^*$-algebra $A$ and an action $\alpha$ of
$\mathbb{N}^{2}$ on $A$ by automorphisms. Let $\pi$ be a nondegenerate representation of $A$ on a Hilbert space $H$. If the maps
$$\tilde{\pi}:A\rightarrow B(\ell^2(\N^{2})\otimes H)\ \ \textrm{and}\ \ V:\N^{2}\rightarrow B(\ell^2(\N^{2})\otimes H)$$
are defined by
$$(\tilde{\pi}(a)f)(s)=\pi(\alpha_{s}(a))f(s)\ \ \textrm{and}\ \ (V_{t}f)(s)=f(s+t)$$
for all $f\in \ell^2(\N^{2})\otimes H\simeq \ell^2(\N^{2},H)$ and $s,t\in \N^{2}$, then the pair $(\tilde{\pi},V)$ is a covariant
partial-isometric representation of $(A,\N^{2},\alpha)$ on $\ell^2(\N^{2},H)$.
Moreover, if $\pi$ is faithful, so is $\tilde{\pi}$ (see \cite[Example 4.3]{SZ3}).

Let $A\times_{\alpha}^{\piso}\N^{2}$ be the partial-isometric crossed product of the system $(A,\mathbb{N}^{2},\alpha)$. Recall from
\cite[\S 2]{SZ2} (see also \cite[Remark 3.11]{SZ4}) that we have the short exact sequence
\begin{align}
\label{ext.seq.1}
0 \longrightarrow \ker q \stackrel{}{\longrightarrow} A\times_{\alpha}^{\piso}\N^{2}
\stackrel{q}{\longrightarrow} A\rtimes_{\alpha}\Z^{2} \longrightarrow 0
\end{align}
of $C^*$-algebras, and by \cite[Corollary 3.13]{SZ4}, the algebra $\K(\ell^{2}(\N^{2}))\otimes A$ of compact operators is
contained in $\ker q$ as an (essential) ideal such that
$$\ker q/[\K(\ell^{2}(\N^{2}))\otimes A]\simeq
\big[\K(\ell^{2}(\N))\otimes (A\rtimes_{\dot{\alpha}}\Z)\big] \oplus \big[\K(\ell^{2}(\N))\otimes (A\rtimes_{\ddot{\alpha}}\Z)\big],$$
where $\dot{\alpha}$ and $\ddot{\alpha}$ are two automorphic actions corresponding to two generators of the group $\Z^{2}$.
These results in \cite{SZ4} are obtained by applying the fact that the algebra $A\times_{\alpha}^{\piso}\N^{2}$ is a full corner
in a crossed product by the group $\Z^{2}$, which is provided in \cite{SZ2}. More precisely, let $B_{\Z^{2}}$ be the $C^{*}$-subalgebra of $\ell^{\infty}(\Z^{2})$ generated by the characteristic functions $\{1_{x}\in\ell^{\infty}(\Z^{2}):x\in \Z^{2}\}$, such that
\[
1_{x}(y)=
   \begin{cases}
      1 &\textrm{if}\empty\ \text{$x\leq y$,}\\
      0 &\textrm{otherwise.}
   \end{cases}
\]
Then, there is an action $\tau$ of $\Z^{2}$ on $B_{\Z^{2}}$ given by translation. So, the system $(A,\mathbb{N}^{2},\alpha)$ gives rise
to the group dynamical system $(B_{\Z^{2}}\otimes A, \Z^{2}, \tau\otimes \alpha^{-1})$. Also, if $B_{\Z^{2},\infty}$ is the $C^{*}$-subalgebra
of $B_{\Z^{2}}$ generated by the elements $\{1_{x}-1_{y}: x\leq y\in \Z^{2}\}$, then it is a $\tau$-invarian (essential) ideal of $B_{\Z^{2}}$.
Now, by \cite[Corollary 5.3]{SZ2}, $A\times_{\alpha}^{\piso} \N^{2}$ and the ideal $\ker q$ sit in the group crossed products
$(B_{\Z^{2}}\otimes A)\times_{\tau\otimes\alpha^{-1}} \Z^{2}$ and $(B_{\Z^{2},\infty}\otimes A)\times_{\tau\otimes\alpha^{-1}} \Z^{2}$,
respectively, as full corners. Thus, the information on $A\times_{\alpha}^{\piso} \N^{2}$ in \cite{SZ4} are indeed imported from the
group crossed product $(B_{\Z^{2}}\otimes A)\times_{\tau\otimes\alpha^{-1}} \Z^{2}$.

\subsection{The primitive ideal space of crossed products by groups}
\label{prim ideal LCTG}
Suppose that $G$ is an abelian countable discrete group which acts on a second countable locally compact Hausdorff space
$X$. So, the pair $(G,X)$ is a second countable locally compact transformation group, which gives rise to the separable
group dynamical system $(C_{0}(X),G,\lt)$ with $G$ abelian. If $C_{0}(X)\rtimes_{\lt} G$ is the group crossed product of
the system, then its primitive ideals are known by \cite[Theorem 8.21]{W}, and a complete description of the topology of
$\Prim(C_{0}(X)\rtimes_{\lt} G)$ is available in \cite[Theorem 8.39]{W}. In brief, for every $x\in X$, let
$$\varepsilon_{x}:C_{0}(X)\rightarrow \mathbb{C}$$
be the evaluation map at $x$, and the sets
$$G\cdot x:=\{t\cdot x: t\in G\}\ \ \textrm{and}\ \ G_{x}:=\{t\in G: t\cdot x=x\}$$
the \emph{$G$-orbit} and the \emph{stability group} of $x$, respectively. Now, there is an equivalence relation on the product space
$X\times \widehat{G}$ such that $(x,\gamma)\sim (y,\mu)$ if
\begin{align}
\label{f2}
\overline{G\cdot x}=\overline{G\cdot y}\ (\textrm{which implies that}\ G_{x}=G_{y})\ \ \textrm{and}\ \ \gamma|_{G_{x}}=\mu|_{G_{x}}.
\end{align}
Let $X\times \widehat{G}/\sim$ be the quotient space equipped with the quotient topology. We have:
\begin{theorem}\cite[Theorem 8.39]{W}
\label{Will th}
Let $(G,X)$ be a second countable locally compact transformation group with $G$ abelian. Then, the map
$$\Phi:X\times \widehat{G} \rightarrow \Prim(C_{0}(X)\rtimes_{\lt} G)$$
defined by
$$\Phi(x,\gamma):=\ker\big(\Ind_{G_{x}}^{G}(\varepsilon_{x}\rtimes \gamma|_{G_{x}})\big)$$
is a continuous and open surjection which factors through a homeomorphism of $X\times \widehat{G}/\sim$ onto $\Prim(C_{0}(X)\rtimes_{\lt} G)$.
\end{theorem}
To see more details, interested readers are referred to \cite{W}.

Next, recall that if $G$ is a (discrete) group with the unit element $e$ which acts on a topological space $Z$, then the action
of $G$ on $Z$ is called \emph{free} (or we say $G$ acts on $Z$ \emph{freely}) if all stability groups are just the trivial subgroup $\{e\}$.
Also, there is an equivalence relation $\sim$ on $Z$ such that
$$z_{1}\sim z_{2}\Longleftrightarrow\overline{G\cdot z_{1}}=\overline{G\cdot z_{2}}.$$
for all $z_{1},z_{2}\in Z$. If $\O(Z)$ denotes the set of all equivalence classes, then it is called the \emph{quasi-orbit space} when
equipped with the quotient topology, which is always a $T_{0}$-topological space. The equivalence class of each $z\in Z$ is called the \emph{quasi-orbit} of $z$ and denoted by $\O(z)$. As an example, if $(A,G,\alpha)$ is a group dynamical system, then we can talk about the
quasi-orbit space $\O(\Prim A)$. This is due to the fact that the system defines an action of $G$ on $\Prim A$ by
$$t\cdot P:=\alpha_{t}(P)=\{\alpha_{t}(a): a\in P \}$$
for all $t\in G$ and $P\in \Prim A$.

Let $(A,G,\alpha)$ be a group dynamical system, $\pi$ a nondegenerate representation of $A$ on a Hilbert space $H$ such that $\ker\pi=J$.
Recall that there is a covariant representation $(\tilde{\pi},U)$ of $(A,G,\alpha)$ on the Hilbert space
$\ell^{2}(G,H)\simeq \ell^{2}(G)\otimes H$ defined by
$$(\tilde{\pi}(a)f)(s)=\pi(\alpha_{s^{-1}}(a))f(s)\ \ \textrm{and}\ \ (U_{t}f)(s)=f(t^{-1}s)$$ for every $a\in A$, $f\in \ell^{2}(G,H)$,
and $s,t\in G$. The corresponding (nondegenerate) representation $\tilde{\pi}\rtimes U$ of the crossed product $A\rtimes_{\alpha} G$ of
the system is denoted by $\Ind\pi$, and $\ker(\Ind \pi)$ by $\Ind J=\Ind (\ker \pi)$. Now, if in the system $(A,G,\alpha)$, $A$ is separable
and $G$ is an abelian discrete countable group, which acts on $\Prim A$ freely, then each primitive idea of $A\rtimes_{\alpha} G$ is of the
form $\Ind P$ induced by a primitive ideal $P$ of $A$. More precisely:
\begin{theorem}\cite[Corollary 7.35]{ECH}
\label{prim free action}
Let $(A,G,\alpha)$ be a dynamical system in which $A$ is separable and $G$ is an amenable discrete countable group. If $G$ acts on
$\Prim A$ freely, then the map
$$\O(\Prim A)\rightarrow \Prim(A\rtimes_{\alpha} G)$$
$$\O(P)\mapsto \Ind P=\ker(\Ind \pi)$$
is a homeomorphism, where $\pi$ is an irreducible representation of $A$ with $\ker\pi=P$. In particular, $A\rtimes_{\alpha} G$ is simple
if and only if every $G$-orbit is dense in $\Prim A$, and $A\rtimes_{\alpha} G$ is primitive if and only if there exists a dense
$G$-orbit in $\Prim A$.
\end{theorem}
Recall that we say a $C^{*}$-algebra is \emph{simple} if it does not have any nontrivial ideal. It is called \emph{primitive} if it
has a faithful nonzero irreducible representation, in other words, the zero ideal is a primitive ideal of it.

\section{Primitive Ideals of $A\times_{\alpha}^{\piso}\N^{2}$ derived from $\Prim A$}
\label{prim A}
Let $(A,\mathbb{N}^{2},\alpha)$ be a dynamical system consisting of a $C^*$-algebra $A$ and an action $\alpha$ of $\mathbb{N}^{2}$ on $A$
by automorphisms. Suppose that $A\times_{\alpha}^{\piso}\N^{2}$ is the partial-isometric crossed product of the system. Since it is a
full corner in the group crossed product $(B_{\Z^{2}}\otimes A)\rtimes_{\tau\otimes\alpha^{-1}}\Z^{2}$ by \cite[Corollary 5.3]{SZ2}, in order to
describe $\Prim (A\times_{\alpha}^{\piso}\N^{2})$, it is enough to describe $\Prim ((B_{\Z^{2}}\otimes A)\rtimes_{\tau\otimes\alpha^{-1}}\Z^{2})$
and its topology. First of all, since the algebra $\K(\ell^{2}(\N^{2}))\otimes A$ of compact operators sits in
$A\times_{\alpha}^{\piso}\N^{2}$ as an essential ideal (see \cite[Corollary 3.13]{SZ4}),
$$\Prim (\K(\ell^{2}(\N^{2}))\otimes A)\simeq \Prim A$$
sits in $\Prim (A\times_{\alpha}^{\piso}\N^{2})$ as an open dense subset. More precisely, there is a
homeomorphism of $\Prim A$ onto the open dense subset
$$U:=\big\{\I\in\Prim (A\times_{\alpha}^{\piso}\N^{2}):\K(\ell^{2}(\N^{2}))\otimes A\not\subset \I\big\}$$
of $\Prim (A\times_{\alpha}^{\piso}\N^{2})$. So, we first identify the elements of $U$, namely, the primitive ideals of
$A\times_{\alpha}^{\piso}\N^{2}$ derived from $\Prim A$. Then, in order to identify other primitive ideals of $A\times_{\alpha}^{\piso}\N^{2}$
derived from the other three sets
$$\Prim(A\rtimes_{\dot{\alpha}}\Z), \Prim(A\rtimes_{\ddot{\alpha}}\Z),\ \textrm{and}\ \Prim (A\rtimes_{\alpha}\Z^{2})
\ (\textrm{see}\ \S\ref{sec:pre}\ \textrm{or}\ \S\ref{intro}),$$
and describe the topology of
$$\Prim ((B_{\Z^{2}}\otimes A)\rtimes_{\tau\otimes\alpha^{-1}}\Z^{2})\simeq \Prim (A\times_{\alpha}^{\piso}\N^{2}),$$
we will consider some conditions on the system $(A,\mathbb{N}^{2},\alpha)$ so that we can apply the results of \cite{ECH,W}.

Now, the following proposition identifies the elements of $U$.
\begin{prop}
\label{Ip}
Let $\pi:A\rightarrow B(H)$ be a nonzero irreducible representation of $A$ such that $P=\ker \pi$. If
$(\tilde{\pi},V)$ is the pair defined in \cite[Example 4.3]{SZ3} (see \S\ref{sec:pre}), then the corresponding representation
$\tilde{\pi}\times V$ of $(A\times_{\alpha}^{\piso}\N^{2},i_{A},i_{\N^{2}})$ is irreducible on $\ell^2(\N^{2})\otimes H$ which lives on $\K(\ell^{2}(\N^{2}))\otimes A$.
\end{prop}

\begin{proof}
We first show that the restriction of $\tilde{\pi}\times V$ to the (essential) ideal
$\K(\ell^{2}(\N^{2}))\otimes A\simeq \K(\ell^{2}(\N^{2})\otimes A)$ is the representation $\id\otimes\pi$. Let $\{e_{(m,n)}: (m,n)\in\N^{2}\}$ be the usual orthonormal
basis of $\ell^2(\N^{2})$, and $\xi_{(m,n)}^{(x,y)}(a)$ denote the element
$$i_{\N^{2}}(m,n)^{*} i_{A}(a) [1-i_{\N^{2}}(1,0)^{*}i_{\N^{2}}(1,0)]i_{\N^{2}}(x,y)$$
of the algebra $A\times_{\alpha}^{\piso}\N^{2}$. Recall that, by \cite[Lemma 3.8]{SZ4}, the elements of the form
\begin{align}
\label{span-K}
\xi_{(m,n)}^{(x,y)}(a)-\xi_{(m,n+1)}^{(x,y+1)}(\alpha_{(0,1)}(a))
\end{align}
span an (essential) ideal $L$ of $A\times_{\alpha}^{\piso}\N^{2}$ which is isomorphic to the algebra $\K(\ell^{2}(\N^{2}))\otimes A$ of compact operators via an isomorphism $\varphi$, such that
$$\varphi\big((e_{(m,n)}\otimes\overline{e_{(x,y)}})\otimes ab^{*}\big)
=\xi_{(m,n)}^{(x,y)}(ab^{*})-\xi_{(m,n+1)}^{(x,y+1)}(\alpha_{(0,1)}(ab^{*}))$$
for all $a,b\in A$ and $(m,n),(x,y)\in\N^{2}$ (see \cite[Theorem 3.10]{SZ4}), where $(e_{(m,n)}\otimes\overline{e_{(x,y)}})$ is the rank-one operator on $\ell^{2}(\N^{2})$ defined by
$g\mapsto \langle g|e_{(x,y)}\rangle e_{(m,n)}$. Now, by calculation on spanning elements, we have
\begin{eqnarray}
\label{eq6}
\begin{array}{l}
(\tilde{\pi}\times V)\bigg(\xi_{(m,n)}^{(x,y)}(ab^{*})-\xi_{(m,n+1)}^{(x,y+1)}(\alpha_{(0,1)}(ab^{*}))\bigg)(e_{(r,s)}\otimes h)\\
=[(\tilde{\pi}\times V)\circ\varphi]\big((e_{(m,n)}\otimes\overline{e_{(x,y)}})\otimes ab^{*}\big)(e_{(r,s)}\otimes h)\\
=(\id\otimes\pi)\big((e_{(m,n)}\otimes\overline{e_{(x,y)}})\otimes ab^{*}\big)(e_{(r,s)}\otimes h)
\end{array}
\end{eqnarray}
for all $a,b\in A$, $(m,n),(x,y),(r,s)\in\N^{2}$, and $h\in H$. Thus, $(\tilde{\pi}\times V)|_{\K(\ell^{2}(\N^{2}))\otimes A}=\id\otimes\pi$, from which, since the representation $\pi$ is nonzero,
it follows that $\tilde{\pi}\times V$ lives on the ideal $\K(\ell^{2}(\N^{2}))\otimes A$.

At last, to see that the representation $\tilde{\pi}\times V$ is irreducible, let $f$ be a nonzero vector in $\ell^2(\N^{2})\otimes H\simeq\ell^2(\N^{2},H)$. So, there is $(x,y)\in\N^{2}$ such that
$f(x,y)\neq 0$, and hence, since the representation $\pi$ is irreducible, the nonzero vector $f(x,y)$ of $H$ is a cyclic vector for $\pi$. Thus, it follows that
\begin{align}
\label{span2}
\ell^2(\N^{2})\otimes H=\clsp\big\{ e_{(m,n)}\otimes \big(\pi(a)f(x,y)\big): (m,n)\in\N^{2}, a\in A\big\}.
\end{align}
However, we have
$$e_{(m,n)}\otimes \big(\pi(a)f(x,y)\big)
=(\tilde{\pi}\times V)\bigg(\xi_{(m,n)}^{(x,y)}(a)-\xi_{(m,n+1)}^{(x,y+1)}(\alpha_{(0,1)}(a))\bigg)f\ \ (\textrm{see}\ (\ref{eq6})),$$
which implies that
\begin{align}
\label{eq1}
\ell^2(\N^{2})\otimes H=\clsp\big\{ (\tilde{\pi}\times V)(\xi)f: \xi\in A\times_{\alpha}^{\piso}\N^{2} \big\}.
\end{align}
So, every nonzero vector $f$ of $\ell^2(\N^{2})\otimes H$ is a cyclic vector for $\tilde{\pi}\times V$, and therefore, $\tilde{\pi}\times V$ is irreducible. This completes the proof.
\end{proof}

\begin{remark}
\label{map Ip}
Therefore, by Proposition \ref{Ip}, each element of $U$ is the kernel of an irreducible representation $\tilde{\pi}\times V$
induced by a primitive ideal $P=\ker \pi$ of $A$. So, we denote $\ker (\tilde{\pi}\times V)$ by $\ind P$, and hence,
$$U=\{\ind P: P\in\Prim A\},$$
which is homeomorphic to $\Prim A$ via the homeomorphism $P\mapsto \ind P$. This homeomorphism is obtained by the composition of
homeomorphisms
$$\ind P \in U\mapsto \ind P\cap (\K(\ell^{2}(\N^{2}))\otimes A)\in \Prim (\K(\ell^{2}(\N^{2}))\otimes A),$$
where
\begin{eqnarray*}
\begin{array}{rcl}
\ind P\cap (\K(\ell^{2}(\N^{2}))\otimes A)&=&\ker((\tilde{\pi}\times V)|_{\K(\ell^{2}(\N^{2}))\otimes A})\\
&=&\ker(\id\otimes\pi)=\K(\ell^{2}(\N^{2}))\otimes P,
\end{array}
\end{eqnarray*}
and
$$P\in \Prim A\mapsto \K(\ell^{2}(\N^{2}))\otimes P\in \Prim (\K(\ell^{2}(\N^{2}))\otimes A)\ \ \textrm{(the Rieffel homeomorphism)}.$$
\end{remark}

\begin{remark}
\label{simple}
One can immediately see that $A\times_{\alpha}^{\piso}\N^{2}$ is not simple as it contains the algebra
$\K(\ell^{2}(\N^{2}))\otimes A$ as a proper nonzero ideal and $A\neq0$.
\end{remark}

\section{The topology of $\Prim (A\times_{\alpha}^{\piso}\N^{2})$ when $A$ is abelian and separable}
\label{case 1}
First, recall that $\widehat{\Z}$, the dual of the group $\Z$, is isomorphic to $\TT$ via the map
$z\in\TT\mapsto \gamma_{z}\in \widehat{\Z}$, such that $\gamma_{z}(n)=z^{n}$ for all $n\in \Z$. Therefore,
$\widehat{\Z^{2}}\simeq \TT^{2}$ via the isomorphism $(z_{1},z_{2})\in\TT^{2}\mapsto \gamma_{(z_{1},z_{2})}\in \widehat{\Z^{2}}$, such that
$$\gamma_{(z_{1},z_{2})}(m,n)=z_{1}^{m}z_{2}^{n}=\gamma_{z_{1}}(m)\gamma_{z_{2}}(n).$$

Now, if in the system $(A,\mathbb{N}^{2},\alpha)$, $A$ is abelian and separable, then
$(B_{\Z^{2}}\otimes A)\rtimes_{\tau\otimes\alpha^{-1}}\Z^{2}$ is isomorphic to the crossed product $C_{0}(X)\rtimes_{\lt} \Z^{2}$
associated with the second countable locally compact transformation group $(\Z^{2},X)$, where $X=\Delta(B_{\Z^{2}}\otimes A)$ is the
spectrum of the (abelian) algebra $B_{\Z^{2}}\otimes A$.
Thus, by Theorem \ref{Will th},
$$\Prim ((B_{\Z^{2}}\otimes A)\rtimes_{\tau\otimes\alpha^{-1}}\Z^{2})\simeq \Prim (A\times_{\alpha}^{\piso}\N^{2})$$
is homeomorphic to the quotient space $X\times \TT^{2}/\sim$.
In order to get a precise description of $X\times \TT^{2}/\sim$, firstly, since
$$B_{\Z^{2}}=B_{(\Z\times \Z)}\simeq B_{\Z}\otimes B_{\Z}\ \ \big(\textrm{this isomorphisms intertwines the actions}\ \tau\ \textrm{and}\ (\lt\otimes\lt)\big),$$
it follows by \cite[Theorem B.45]{RW} (or \cite[Theorem B.37]{RW}) that
\begin{align}
\label{X-sptrm}
X\simeq \Delta(B_{\Z}) \times \Delta(B_{\Z}) \times \Delta(A).
\end{align}
Moreover, by \cite[Lemma 3.3]{LZ}, $\Delta(B_{\Z})$ is homeomorphic to the open dense subset $\mathbb{Z} \cup \{\infty\}$ of the
two-point compactification $\{-\infty\} \cup \mathbb{Z} \cup \{\infty\}$ of $\Z$. Therefore, we actually need to describe
\begin{align}
\label{QS}
\big[(\Z \cup \{\infty\}) \times (\Z \cup \{\infty\}) \times \Delta(A)\big] \times \TT^{2}/\sim
\end{align}
To do so, we first need to see that how the group $\Z^{2}$ acts on the product space
\begin{align}
\label{prod}
X\simeq (\Z \cup \{\infty\}) \times (\Z \cup \{\infty\}) \times \Delta(A).
\end{align}
For every $(m,n)\in \Z^{2}$, $\phi\in \Delta(A)$, and $r,s\in\Z$, we have
$$(m,n)\cdot ((r,s),\phi)=((r+m,s+n),\phi).$$
This is due to the fact that $((r,s),\phi)$ is an element of the spectrum of the (essential) ideal
$$C_{0}(\Z^{2})\otimes A\simeq C_{0}(\Z)\otimes C_{0}(\Z)\otimes A$$
of the algebra $B_{\Z^{2}}\otimes A$, which is invariant under the action $\tau\otimes\alpha^{-1}$. Therefore, the
crossed product $(C_{0}(\Z^{2})\otimes A)\rtimes_{\tau\otimes\alpha^{-1}}\Z^{2}$ sits in the algebra
$(B_{\Z^{2}}\otimes A)\rtimes_{\tau\otimes\alpha^{-1}}\Z^{2}$ as an (essential) ideal (see \cite[Theorem 3.7]{SZ4}).
Furthermore, by \cite[Lemma 7.4]{W}, we have
$$(C_{0}(\Z^{2})\otimes A)\rtimes_{\tau\otimes\alpha^{-1}}\Z^{2}
\simeq (C_{0}(\Z^{2})\otimes A)\rtimes_{\tau\otimes\id}\Z^{2}\simeq \K(\ell^{2}(\Z^{2}))\otimes A.$$

Next, for every $(m,n)\in \Z^{2}$ and $\phi\in \Delta(A)$,
$$(m,n)\cdot ((\infty,\infty),\phi)=((\infty+m,\infty+n),\phi\circ\alpha_{(m,n)})
=((\infty,\infty),\phi\circ\alpha_{(m,n)}).$$

Finally, to see that how $\Z^{2}$ acts on the elements of the forms $((r,\infty),\phi)$ and $((\infty,r),\phi)$, where $r\in \Z$ and
$\phi\in \Delta(A)$, first note that the action $\alpha$ of $\Z^{2}$ induces two actions $\dot{\alpha}$ and $\ddot{\alpha}$
(corresponding to two generators of the group $\Z^{2}$) of (the subgroup) $\Z$ on $A$ by automorphisms, such that
$$\dot{\alpha}_{n}:=\alpha_{(n,0)}\ \ \textrm{and}\ \ \ddot{\alpha}_{n}:=\alpha_{(0,n)}$$
for every $n\in\Z$. Thus, there are two group crossed products $A\rtimes_{\dot{\alpha}}\Z$ and $A\rtimes_{\ddot{\alpha}}\Z$,
correspondingly. The primitive ideal spaces of them are quotients of the space $\Delta(A)\times \TT$ which we denote them by
$\Delta(A)\times \TT/\sim^{(1)}$ and $\Delta(A)\times \TT/\sim^{(2)}$, respectively.
Now, $((r,\infty),\phi)$ is an element of the spectrum of the (essential) ideal
$$(C_{0}(\Z)\otimes B_{\Z})\otimes A$$
of $B_{\Z^{2}}\otimes A\simeq (B_{\Z}\otimes B_{\Z})\otimes A$, which is indeed invariant under the action $\tau\otimes\alpha^{-1}$.
Moreover, there is an isomorphism
$$\psi:(C_{0}(\Z)\otimes B_{\Z})\otimes A\rightarrow C_{0}(\Z)\otimes (B_{\Z}\otimes A)$$
such that
$$((1_{x}-1_{x+1})\otimes 1_{y})\otimes a\mapsto (1_{x}-1_{x+1})\otimes (1_{y}\otimes \dot{\alpha}_{x}(a))$$
for all $x,y\in\Z$ and $a\in A$. By inspection on spanning elements, one can see that for every $(m,n)\in \Z^{2}$
the following diagram commutes:
\begin{equation*}
\begin{diagram}\dgARROWLENGTH=0.4\dgARROWLENGTH
\node{(C_{0}(\Z)\otimes B_{\Z})\otimes A} \arrow{s,l}{\tau_{(m,n)}\otimes\alpha_{(-m,-n)}}\arrow{e,t}{\psi}
\arrow{s}\arrow{e}\node{C_{0}(\Z)\otimes (B_{\Z}\otimes A)}
\arrow{s,r}{\lt_{m}\otimes(\lt_{n}\otimes \ddot{\alpha}_{-n})}\\
\node{(C_{0}(\Z)\otimes B_{\Z})\otimes A} \arrow{e,t}{\psi}
\node{C_{0}(\Z)\otimes (B_{\Z}\otimes A).}
\end{diagram}
\end{equation*}
Therefore, it follows that
$$(m,n)\cdot ((r,\infty),\phi)=((r+m,\infty+n),\phi\circ\ddot{\alpha}_{n})=((r+m,\infty),\phi\circ\alpha_{(0,n)})$$
for every $(m,n)\in \Z^{2}$, $r\in\Z$, and $\phi\in \Delta(A)$. Note that also, by \cite[Lemma 2.65]{W}, the isomorphism $\psi$
induces an isomorphism of the (essential) ideal
$$\big((C_{0}(\Z)\otimes B_{\Z})\otimes A\big)\rtimes_{\tau\otimes\alpha^{-1}}\Z^{2}$$
of $(B_{\Z^{2}}\otimes A)\rtimes_{\tau\otimes\alpha^{-1}}\Z^{2}$ onto the algebra
\begin{eqnarray*}
\begin{array}{l}
\big(C_{0}(\Z)\otimes (B_{\Z}\otimes A)\big)\rtimes_{\lt\otimes(\lt\otimes\ddot{\alpha}^{-1})}(\Z\times \Z)\\
\simeq \big[C_{0}(\Z)\rtimes_{\lt} \Z\big]\otimes\big[(B_{\Z}\otimes A)\rtimes_{\lt\otimes\ddot{\alpha}^{-1}}\Z\big]\\
\simeq \K(\ell^{2}(\Z))\otimes \big[(B_{\Z}\otimes A)\rtimes_{\lt\otimes\ddot{\alpha}^{-1}}\Z\big]
\end{array}
\end{eqnarray*}
of compact operators (see also \cite[Theorem 3.7]{SZ4} and \cite[Remark 3.11]{SZ4}). A similar discussion shows that
$$(m,n)\cdot ((\infty,r),\phi)=((\infty+m,r+n),\phi\circ\dot{\alpha}_{m})=((\infty,r+n),\phi\circ\alpha_{(m,0)})$$
for every $(m,n)\in \Z^{2}$, $r\in\Z$, and $\phi\in \Delta(A)$.

\begin{lemma}
\label{QS-lem}
The quotient space (\ref{QS}), as a set, is identified by the disjoint union of four sets
\begin{align}
\label{QS2}
\Delta(A)\sqcup \Prim(A\rtimes_{\dot{\alpha}}\Z)\sqcup \Prim(A\rtimes_{\ddot{\alpha}}\Z)\sqcup \Prim (A\rtimes_{\alpha}\Z^{2}),
\end{align}
where $\Prim (A\rtimes_{\alpha}\Z^{2})$ is a quotient of the product space $\Delta(A)\times \TT^{2}$.
\end{lemma}

\begin{proof}
For every $r,s\in\Z$ and $\phi\in \Delta(A)$, the stability group of the element $((r,s),\phi)$ of (\ref{prod}) is the trivial subgroup
$\{(0,0)\}$, and its $\Z^{2}$-orbit is $\Z^{2}\times \{\phi\}=\Z\times\Z\times \{\phi\}$. So, the element $((r,s),\phi,(z_{1},z_{2}))$ of
the product space
\begin{align}
\label{prod2}
X\times \TT^{2}\simeq \big[(\Z \cup \{\infty\}) \times (\Z \cup \{\infty\}) \times \Delta(A)\big] \times \TT^{2}
\end{align}
can only be equivalent to an element $((m,n),\psi,(w_{1},w_{2}))$ of the same type, and since $\Delta(A)$ is Hausdorff, we have
\begin{eqnarray*}
\begin{array}{rcl}
((r,s),\phi,(z_{1},z_{2})) \sim ((m,n),\psi,(w_{1},w_{2})) &\Longleftrightarrow&
\overline{\Z^{2} \cdot ((r,s),\phi)}=\overline{\Z^{2} \cdot ((m,n),\psi)}\\
&\Longleftrightarrow&\overline{\Z^{2}\times \{\phi\}}=\overline{\Z^{2}\times \{\psi\}}\\
&\Longleftrightarrow&\overline{\Z^{2}}\times \overline{\{\phi\}}=\overline{\Z^{2}}\times \overline{\{\psi\}}\\
&\Longleftrightarrow&\overline{\Z^{2}}\times \{\phi\}=\overline{\Z^{2}}\times \{\psi\}.\\
\end{array}
\end{eqnarray*}
It thus follows that
$$((r,s),\phi,(z_{1},z_{2})) \sim ((m,n),\psi,(w_{1},w_{2})) \Longleftrightarrow \phi=\psi.$$
Therefore, all elements $((r,s),\phi,(z_{1},z_{2}))$ in which $\phi\in \Delta(A)$ is fixed and $(r,s)$ and $(z_{1},z_{2})$
are running in $\Z^{2}$ and $\TT^{2}$, respectively, are in the same equivalence class in (\ref{QS}), which can be
parameterized by $\phi \in \Delta(A)$.

Next, the stability group of the element $((\infty,\infty),\phi)$ of (\ref{prod}) equals the stability group $\Z^{2}_{\phi}$ of $\phi$
when $\Z^{2}$ acts on $\Delta(A)$ via the action $\alpha$ (corresponding to the crossed product $A\rtimes_{\alpha}\Z^{2}$).
Its $\Z^{2}$-orbit is $\{(\infty,\infty)\}\times (\Z^{2}\cdot\phi)$, where $\Z^{2}\cdot \phi$ is the $\Z^{2}$-orbit of $\phi$.
Moreover, the element $((\infty,\infty),\phi,(z_{1},z_{2}))$ of (\ref{prod2}) can only be equivalent to an element $((\infty,\infty),\psi,(w_{1},w_{2}))$ of the same type, and we have
\begin{eqnarray*}
\begin{array}{l}
((\infty,\infty),\phi,(z_{1},z_{2}))\sim ((\infty,\infty),\psi,(w_{1},w_{2}))\\
\Longleftrightarrow \overline{\Z^{2} \cdot((\infty,\infty),\phi)}=\overline{\Z^{2}\cdot ((\infty,\infty),\psi)}
\ \textrm{and}\ \gamma_{(z_{1},z_{2})}|_{\Z^{2}_{\phi}}=\gamma_{(w_{1},w_{2})}|_{\Z^{2}_{\phi}}\\
\Longleftrightarrow \overline{\{(\infty,\infty)\}\times (\Z^{2}\cdot\phi)}
=\overline{\{(\infty,\infty)\}\times (\Z^{2}\cdot\psi)}\ \textrm{and}\ \gamma_{(z_{1},z_{2})}|_{\Z^{2}_{\phi}}=\gamma_{(w_{1},w_{2})}|_{\Z^{2}_{\phi}}\\
\Longleftrightarrow \overline{\{(\infty,\infty)\}}\times \overline{\Z^{2}\cdot\phi}
=\overline{\{(\infty,\infty)\}}\times \overline{\Z^{2}\cdot\psi}\ \ \textrm{and}\ \gamma_{(z_{1},z_{2})}|_{\Z^{2}_{\phi}}=\gamma_{(w_{1},w_{2})}|_{\Z^{2}_{\phi}}\\
\Longleftrightarrow \{(\infty,\infty)\}\times \overline{\Z^{2}\cdot\phi}
=\{(\infty,\infty)\}\times \overline{\Z^{2}\cdot\psi}\ \ \textrm{and}\ \gamma_{(z_{1},z_{2})}|_{\Z^{2}_{\phi}}=\gamma_{(w_{1},w_{2})}|_{\Z^{2}_{\phi}}.\\
\end{array}
\end{eqnarray*}
Therefore,
\begin{eqnarray*}
\begin{array}{l}
((\infty,\infty),\phi,(z_{1},z_{2}))\sim ((\infty,\infty),\psi,(w_{1},w_{2}))\\
\Longleftrightarrow \overline{\Z^{2}\cdot\phi}=\overline{\Z^{2}\cdot\psi}\ \ \textrm{and}
\ \gamma_{(z_{1},z_{2})}|_{\Z^{2}_{\phi}}=\gamma_{(w_{1},w_{2})}|_{\Z^{2}_{\phi}},
\end{array}
\end{eqnarray*}
which implies that $((\infty,\infty),\phi,(z_{1},z_{2}))\sim ((\infty,\infty),\psi,(w_{1},w_{2}))$ precisely when the pairs
$(\phi,(z_{1},z_{2}))$ and $(\psi,(w_{1},w_{2}))$ are in the same equivalence class in the quotient space
$\Delta(A)\times \TT^{2}/\sim=\Prim (A\rtimes_{\alpha}\Z^{2})$. Consequently, we can parameterize the equivalence class of each element $((\infty,\infty),\phi,(z_{1},z_{2}))$ in (\ref{QS}) by the class of the pair $(\phi,(z_{1},z_{2}))$ in $\Prim (A\rtimes_{\alpha}\Z^{2})$.

At last, it is left to discuss on the parametrization of the equivalence classes of the elements of the forms
$((\infty,r),\phi,(z_{1},z_{2}))$ and $((r,\infty),\phi,(z_{1},z_{2}))$ in (\ref{QS}). We only do this for
$((\infty,r),\phi,(z_{1},z_{2}))$ as the parametrization of the other one follows similarly. Firstly, let $\dot{\Z}_{\phi}$ and
$\dot{\Z}\cdot \phi$ denote the stability group and the $\Z$-orbit of $\phi$, respectively, when the group $\Z$ acts on $\Delta(A)$
via the action $\dot{\alpha}$ (corresponding to the crossed product $A\rtimes_{\dot{\alpha}}\Z$). Then, the stability group of
the element $((\infty,r),\phi)$ of (\ref{prod}) is $\dot{\Z}_{\phi}\times \{0\}$, which is isomorphic to $\dot{\Z}_{\phi}$, and
its $\Z^{2}$-orbit is
$$\big(\{\infty\}\times \Z\big)\times \dot{\Z}\cdot \phi.$$
It therefore follows that the element $((\infty,r),\phi,(z_{1},z_{2}))$ can only be equivalent to an element
$((\infty,s),\psi,(w_{1},w_{2}))$ of the same type. Moreover,
\begin{eqnarray*}
\begin{array}{l}
((\infty,r),\phi,(z_{1},z_{2}))\sim ((\infty,s),\psi,(w_{1},w_{2}))\\
\Longleftrightarrow \overline{\Z^{2}\cdot((\infty,r),\phi)}=\overline{\Z^{2}\cdot ((\infty,s),\psi)}\ \ \textrm{and}
\ \gamma_{(z_{1},z_{2})}|_{(\dot{\Z}_{\phi}\times \{0\})}=\gamma_{(w_{1},w_{2})}|_{(\dot{\Z}_{\phi}\times \{0\})}\\
\Longleftrightarrow \overline{\big(\{\infty\}\times \Z\big)\times \dot{\Z}\cdot \phi}
=\overline{\big(\{\infty\}\times \Z\big)\times \dot{\Z}\cdot \psi}\ \ \textrm{and}
\ \gamma_{z_{1}}|_{\dot{\Z}_{\phi}}=\gamma_{w_{1}}|_{\dot{\Z}_{\phi}}\\
\Longleftrightarrow \overline{\{\infty\}\times \Z}\times \overline{\dot{\Z}\cdot \phi}
=\overline{\{\infty\}\times \Z}\times \overline{\dot{\Z}\cdot \psi}\ \ \textrm{and}
\ \gamma_{z_{1}}|_{\dot{\Z}_{\phi}}=\gamma_{w_{1}}|_{\dot{\Z}_{\phi}}\\
\Longleftrightarrow \overline{\dot{\Z}\cdot \phi}=\overline{\dot{\Z}\cdot \psi}\ \ \textrm{and}
\ \gamma_{z_{1}}|_{\dot{\Z}_{\phi}}=\gamma_{w_{1}}|_{\dot{\Z}_{\phi}}.
\end{array}
\end{eqnarray*}
This implies that $((\infty,r),\phi,(z_{1},z_{2}))\sim ((\infty,s),\psi,(w_{1},w_{2}))$ if and only if the pairs $(\phi,z_{1})$ and
$(\psi,w_{1})$ are in the same equivalence class in the quotient space $\Delta(A)\times \TT/\sim^{(1)}=\Prim (A\rtimes_{\dot{\alpha}}\Z)$.
Thus, the equivalence class of each element $((\infty,r),\phi,(z_{1},z_{2}))$ in (\ref{QS}) can be parameterized by the class of the pair
$(\phi,z_{1})$ in $\Prim (A\rtimes_{\dot{\alpha}}\Z)$. Note that a similar discussion shows that each element
$((r,\infty),\phi,(z_{1},z_{2}))$ can only be equivalent to an element of the same type, and its equivalence class in (\ref{QS})
is parameterized by the class of the pair $(\phi,z_{2})$ in $\Delta(A)\times \TT/\sim^{(2)}=\Prim (A\rtimes_{\ddot{\alpha}}\Z)$.
This completes the proof.
\end{proof}

We are now ready to describe the topology of $\Prim (A\times_{\alpha}^{\piso}\N^{2})$ precisely.
\begin{theorem}
\label{prim piso}
Let $(A,\N^{2},\alpha)$ be a dynamical system consisting of a separable abelian $C^{*}$-algebra $A$ and an action $\alpha$ of
$\N^{2}$ on $A$ by automorphisms. Then, $\Prim (A\times_{\alpha}^{\piso}\N^{2})$ is homeomorphic to the disjoint union (\ref{QS2})
equipped with the quotient topology in which the open sets are in the following four forms:
\begin{itemize}
\item[(a)] $O\subset \Delta(A)$, where $O$ is open in $\Delta(A)$;
\item[(b)] $O\cup W_{1}$, where $O$ is a nonempty open subset of $\Delta(A)$ and $W_{1}$ is an open set in $\Prim (A\rtimes_{\dot{\alpha}}\Z)$;
\item[(c)] $O\cup W_{2}$, where $O$ is a nonempty open subset of $\Delta(A)$ and $W_{2}$ is an open set in $\Prim (A\rtimes_{\ddot{\alpha}}\Z)$; and
\item[(d)] $O\cup W_{1}\cup W_{2}\cup W$, where $O$, $W_{1}$, and $W_{2}$ are nonempty open subsets of $\Delta(A)$,
$\Prim (A\rtimes_{\dot{\alpha}}\Z)$, and $\Prim (A\rtimes_{\ddot{\alpha}}\Z)$, respectively, and $W$ is an open set in $\Prim(A\rtimes_{\alpha}\Z^{2})$.
\end{itemize}

\end{theorem}

\begin{proof}
Assume that $\tilde{\textsf{q}}$ is the quotient map of the product space
\begin{align}
\label{prod3}
\big[(\Z \cup \{\infty\}) \times (\Z \cup \{\infty\}) \times \Delta(A)\big] \times (\TT \times \TT)
\end{align}
onto (\ref{QS2}).
Let $\textsf{q}_{1}:\Delta(A)\times \TT \rightarrow \Prim (A\rtimes_{\dot{\alpha}}\Z)$,
$\textsf{q}_{2}:\Delta(A)\times \TT \rightarrow \Prim (A\rtimes_{\ddot{\alpha}}\Z)$, and
$\textsf{q}:\Delta(A)\times \TT^{2} \rightarrow \Prim (A\rtimes_{\alpha}\Z^{2})$ be the quotient maps. Recall that these quotient maps
are all open (see \cite[Remark 8.40]{W}). Let $\widetilde{\mathfrak{B}}$ be the set of all elements
\begin{align}
\label{basis}
U_{1} \times U_{2} \times O \times (V_{1}\times V_{2}),
\end{align}
where each $U_{i}$ is either $\{n_{i}\}$ or $J_{n_{i}}=\{n_{i},n_{i}+1,n_{i}+2,...\}\cup \{\infty\}$ for some $n_{i}\in\Z$
(see \cite[Lemma 3.3]{LZ}), $O$ is an open set in $\Delta(A)$, and each $V_{i}$ is an open subset of $\TT$. Obviously,
$\widetilde{\mathfrak{B}}$ is a basis for the topology of the product space (\ref{prod3}). Therefore, the forward image of the
elements of $\widetilde{\mathfrak{B}}$ by $\tilde{\textsf{q}}$ forms a basis for the quotient topology of (\ref{QS2}) which we denote it by $\mathfrak{B}$. But, first note that, since each $U_{i}$ has two forms, the elements of $\widetilde{\mathfrak{B}}$ have
totally four forms as follow:
\begin{itemize}
\item[(i)] $\{n_{1}\} \times \{n_{2}\} \times O \times (V_{1}\times V_{2})$;
\item[(ii)] $J_{n_{1}} \times \{n_{2}\} \times O \times (V_{1}\times V_{2})$;
\item[(iii)] $\{n_{1}\} \times J_{n_{2}} \times O \times (V_{1}\times V_{2})$; and
\item[(iv)] $J_{n_{1}} \times J_{n_{2}} \times O \times (V_{1}\times V_{2})$.
\end{itemize}
Therefore, accordingly, $\mathfrak{B}$ is the union of the following four sets (see Lemma \ref{QS-lem}):
$$\mathfrak{B}_{1}:=\big\{O\subset\Delta(A): O\ \textrm{is open in}\ \Delta(A)\big\},$$
$$\mathfrak{B}_{2}:=\big\{O\cup \textsf{q}_{1}(O\times V_{1}): O\ \textrm{is a nonempty open subset of}\ \Delta(A),\ \textrm{and}
\ V_{1}\ \textrm{is open in}\ \TT\big\},$$
$$\mathfrak{B}_{3}:=\big\{O\cup \textsf{q}_{2}(O\times V_{2}): O\ \textrm{is a nonempty open subset of}\ \Delta(A),\ \textrm{and}
\ V_{2}\ \textrm{is open in}\ \TT\big\},$$ and
$$\mathfrak{B}_{4}:=\big\{O\cup \textsf{q}_{1}(O\times V_{1})\cup \textsf{q}_{2}(O\times V_{2})\cup \textsf{q}(O\times (V_{1}\times V_{2})):$$
$$O\ \textrm{is a nonempty open subset of}\ \Delta(A),\ \textrm{and each}
\ V_{i}\ \textrm{is a nonempty open subset of}\ \TT\big\}.$$
So, the rest follows from the facts that the open sets $\textsf{q}_{1}(O\times V_{1})$, $\textsf{q}_{2}(O\times V_{2})$, and
$\textsf{q}(O\times (V_{1}\times V_{2}))$ form bases for the topological spaces $\Prim (A\rtimes_{\dot{\alpha}}\Z)$,
$\Prim (A\rtimes_{\ddot{\alpha}}\Z)$, $\Prim (A\rtimes_{\alpha}\Z^{2})$, respectively.
\end{proof}

\begin{remark}
\label{rmk3}
Recall that the primitive ideals of $A\times_{\alpha}^{\piso}\N^{2}$ derived from $\Prim(A\rtimes_{\alpha}\Z^{2})$ form a closed subset
of $\Prim (A\times_{\alpha}^{\piso}\N^{2})$ (see (\ref{ext.seq.1})), which is
$$F:=\{\J\in\Prim (A\times_{\alpha}^{\piso}\N^{2}): \ker q\subset\J\}.$$
Under the conditions of Theorem \ref{prim piso}, these ideals are actually the kernels of the irreducible representations $\big(\Ind_{\Z^{2}_{\phi}}^{\Z^{2}}(\phi \rtimes \gamma_{(z,w)}|_{\Z^{2}_{\phi}})\big)\circ q$ corresponding to the elements (equivalence classes)
$[(\phi,(z,w))]$ of $\Delta(A)\times \TT^{2}/\sim =\Prim (A\rtimes_{\alpha}\Z^{2})$. We denote
$\ker \big([\Ind_{\Z^{2}_{\phi}}^{\Z^{2}}(\phi \rtimes \gamma_{(z,w)}|_{\Z^{2}_{\phi}})]\circ q\big)$ by $\J_{[(\phi,(z,w))]}$, and therefore,
$$F=\{\J_{[(\phi,(z,w))]}: \phi\in\Delta(A), (z,w)\in\TT^{2}\}.$$
Also, the homeomorphism of $\Prim(A\rtimes_{\alpha}\Z^{2})$ onto $F$ is given by the map $$[(\phi,(z,w))]\mapsto \J_{[(\phi,(z,w))]}.$$
\end{remark}

Next, we want to identify the primitive ideals of $A\times_{\alpha}^{\piso}\N^{2}$ derived from $\Prim (A\rtimes_{\dot{\alpha}}\Z)$
and $\Prim (A\rtimes_{\ddot{\alpha}}\Z)$, respectively, under the conditions of Theorem \ref{prim piso}. Consider the semigroup dynamical
system $(A,\N,\dot{\alpha})$, corresponding to which, there is the following short exact sequence
\begin{align}
\label{ext.seq.2}
0 \longrightarrow \K(\ell^{2}(\N))\otimes A \stackrel{}{\longrightarrow} A\times_{\dot{\alpha}}^{\piso}\N
\stackrel{\dot{q}}{\longrightarrow} A\rtimes_{\dot{\alpha}}\Z \longrightarrow 0
\end{align}
of $C^*$-algebras (see \cite{AZ,SZ}). So, corresponding to each element (equivalent class) $[(\phi,z)]$ of
$\Delta(A)\times \TT/\sim^{(1)}\simeq \Prim (A\rtimes_{\dot{\alpha}}\Z)$, the composition
\begin{align}
\label{pi.dot}
\big(\Ind_{\dot{\Z}_{\phi}}^{\Z}(\phi \rtimes \gamma_{z}|_{\dot{\Z}_{\phi}})\big)\circ \dot{q}
\end{align}
gives a nonzero irreducible representation $\dot{\pi}:A\times_{\dot{\alpha}}^{\piso}\N\rightarrow B(H)$ of
$(A\times_{\dot{\alpha}}^{\piso}\N, j_{A},v)$ on a Hilbert space $H$, where $\dot{\Z}_{\phi}$ denotes the stability group
of $\phi$ when the group $\Z$ acts on $\Delta(A)$ via the action $\dot{\alpha}$. It follows that
$\K(\ell^{2}(\N))\otimes (\ker \dot{\pi})$ is a primitive ideal of the algebra $\K(\ell^{2}(\N))\otimes (A\times_{\dot{\alpha}}^{\piso}\N)$,
which by \cite[Corollary 3.12]{SZ4} sits in $A\times_{\alpha}^{\piso}\N^{2}$ as an (essential) ideal $\I_{1}$ (more
precisely, $\I_{1}$ is the ideal $\I_{\gamma}$ in \cite{SZ4}). Also, $\I_{1}$ contains the algebra
$$\K(\ell^{2}(\N^{2}))\otimes A\simeq \K(\ell^{2}(\N)\otimes \ell^{2}(\N))\otimes A
\simeq \K(\ell^{2}(\N))\otimes[\K(\ell^{2}(\N))\otimes A]$$
of compact operators as an (essential) ideal (see again \cite{SZ4}). Recall that the map $T:\N\rightarrow B(\ell^{2}(\N))$ defined by
$T_{n}(e_{m})=e_{m+n}$ on the usual orthonormal basis $\{e_{m}: m\in \N\}$ of $\ell^{2}(\N)$ is a representation of $\N$
by isometries, such that
$$\K(\ell^{2}(\N))=\clsp\{T_{m}(1-TT^{*})T_{n}^{*}: m,n\in \N\}\ \ \ \ (T:=T_{1})$$
Indeed, for every $m,n\in\N$, $T_{m}(1-TT^{*})T_{n}^{*}$ is a rank-one operator on $\ell^{2}(\N)$
such that $f\mapsto \langle f| e_{n}\rangle e_{m}$. Now, the following lemma identifies the primitive ideals of
$A\times_{\alpha}^{\piso}\N^{2}$ coming from $\Prim (A\rtimes_{\dot{\alpha}}\Z)$.

\begin{lemma}
\label{from-alpha1}
Define the maps
$$\dot{\rho}:A\rightarrow B(\ell^{2}(\N)\otimes H)\ \ \textrm{and}\ \ \dot{W}:\N^{2}\rightarrow B(\ell^{2}(\N)\otimes H)$$
by
$$(\dot{\rho}(a)f)(n)=(\dot{\pi}\circ j_{A})(\ddot{\alpha}_{n}(a))f(n)\ \ \textrm{and}
\ \ \dot{W}_{(m,n)}=T_{n}^{*}\otimes \overline{\dot{\pi}}(v_{m}),$$
respectively, for all $a\in A$, $f\in \ell^{2}(\N)\otimes H$, and $m,n\in\N$. Then, the pair $(\dot{\rho},\dot{W})$ is a covariant
partial-isometric representation of the system $(A,\N^{2}, \alpha)$ on the Hilbert space $\ell^{2}(\N)\otimes H\simeq \ell^{2}(\N,H)$,
such that the corresponding (nondegenerate) representation $\dot{\rho}\times \dot{W}$ of $(A\times_{\alpha}^{\piso}\N^{2},i)$
is irreducible on $\ell^{2}(\N)\otimes H$, which lives on the ideal $\I_{1}\simeq \K(\ell^{2}(\N))\otimes (A\times_{\dot{\alpha}}^{\piso}\N)$
but vanishes on $\K(\ell^{2}(\N^{2}))\otimes A$.
\end{lemma}

\begin{proof}
Firstly, some routine calculations (on spanning elements) show that the pair $(\dot{\rho},\dot{W})$ is indeed a covariant
partial-isometric representation of $(A,\N^{2}, \alpha)$ on $\ell^{2}(\N)\otimes H$ which we skip them here.

Next, to see that the corresponding representation
$\dot{\rho}\times \dot{W}:(A\times_{\alpha}^{\piso}\N^{2},i)\rightarrow B(\ell^{2}(\N)\otimes H)$
is irreducible on $\ell^{2}(\N)\otimes H$, we show that every nonzero vector $f\in \ell^{2}(\N)\otimes H$ is a cyclic vector for
$\dot{\rho}\times \dot{W}$. Since $f\neq 0$, there is $y\in\N$ such that $f(y)$ is a nonzero vector in $H$. Therefore, since the
representation $\dot{\pi}$ is irreducible, $f(y)$ is a cyclic vector for $\dot{\pi}$, and hence, the elements
$$\big\{e_{n}\otimes [\dot{\pi}(v_{m}^{*}j_{A}(a)v_{x})f(y)]: a\in A, n,m,x\in\N \big\}$$
span the Hilbert space $\ell^{2}(\N)\otimes H$
(recall that $A\times_{\dot{\alpha}}^{\piso}\N=\clsp\big\{v_{m}^{*}j_{A}(a)v_{x}: a\in A, m,x\in\N\big\}$).
We show the each spanning element $e_{n}\otimes [\dot{\pi}(v_{m}^{*}j_{A}(a)v_{x})f(y)]$ belongs to
$$\clsp\big\{(\dot{\rho}\times \dot{W})(\eta)f: \eta\in A\times_{\alpha}^{\piso}\N^{2}\big\},$$
which implies that $f$ is cyclic for $\dot{\rho}\times \dot{W}$. Take the element
\begin{align}
\label{span_I1}
\eta_{(m,n)}^{(x,y)}(a):=i_{\N^{2}}(m,n)^{*}i_{A}(a)[1-i_{\N^{2}}(0,1)^{*}i_{\N^{2}}(0,1)]i_{\N^{2}}(x,y)
\end{align}
of $A\times_{\alpha}^{\piso}\N^{2}$. See in \cite[Lemma 3.8]{SZ4} that, in fact, the elements of the form (\ref{span_I1}) span
the (essential) ideal $\I_{1}$ of $A\times_{\alpha}^{\piso}\N^{2}$. Now, one can calculate to see that
\begin{eqnarray}\label{eq3}
\begin{array}{l}
(\dot{\rho}\times \dot{W})\bigg(\eta_{(m,n)}^{(x,y)}(a)\bigg)f\\
=\big(T_{n}\otimes \overline{\dot{\pi}}(v_{m}^{*})\big)\dot{\rho}(a)\big[1-(T\otimes 1)(T^{*}\otimes 1)\big]
\big(T_{y}^{*}\otimes \overline{\dot{\pi}}(v_{x})\big)f\\
=\big(T_{n}\otimes \overline{\dot{\pi}}(v_{m}^{*})\big)\dot{\rho}(a)\big[(1-TT^{*})\otimes 1\big]
\big(T_{y}^{*}\otimes \overline{\dot{\pi}}(v_{x})\big)f\\
=\big(T_{n}\otimes \overline{\dot{\pi}}(v_{m}^{*})\big)\dot{\rho}(a)\big[(1-TT^{*})T_{y}^{*}\otimes \overline{\dot{\pi}}(v_{x})\big]f\\
=e_{n}\otimes [\dot{\pi}(v_{m}^{*}j_{A}(a)v_{x})f(y)]
\end{array}
\end{eqnarray}
for all $a\in A$ and $n,m,x\in\N$. It thus follows that $f$ is a cyclic vector for $\dot{\rho}\times \dot{W}$.

To see that the restriction of $\dot{\rho}\times \dot{W}$ to the ideal $\I_{1}\simeq \K(\ell^{2}(\N))\otimes (A\times_{\dot{\alpha}}^{\piso}\N)$
is nonzero, we first show that the restriction $(\dot{\rho}\times \dot{W})|_{\I_{1}}$ is the representation
$$\id\otimes \dot{\pi}:\K(\ell^{2}(\N))\otimes (A\times_{\dot{\alpha}}^{\piso}\N)\rightarrow B(\ell^{2}(\N)\otimes H),$$
such that $(\id\otimes \dot{\pi})(S\otimes \xi)=S\otimes \dot{\pi}(\xi)$ for all $S\in \K(\ell^{2}(\N))$ and
$\xi\in A\times_{\dot{\alpha}}^{\piso}\N$. It is enough to see this on spanning elements, and therefore, we calculate
\begin{eqnarray}\label{eq4}
\begin{array}{l}
(\id\otimes \dot{\pi})\big([T_{n}(1-TT^{*})T_{y}^{*}]\otimes [v_{m}^{*}j_{A}(a)v_{x}]\big)(e_{r}\otimes h)\\
=\big(T_{n}(1-TT^{*})T_{y}^{*}\otimes \dot{\pi}(v_{m}^{*}j_{A}(a)v_{x})\big)(e_{r}\otimes h)\\
=[T_{n}(1-TT^{*})T_{y}^{*} e_{r}]\otimes [\dot{\pi}(v_{m}^{*}j_{A}(a)v_{x}) h]\in \ell^{2}(\N)\otimes H,
\end{array}
\end{eqnarray}
which is equal to $e_{n}\otimes [\dot{\pi}(v_{m}^{*}j_{A}(a)v_{x})h]$ if $r=y$, otherwise, zero for all $m,n,x,y,r\in \N$,
$a\in A$, and $h\in H$. On the other hand, see in \cite[Proposition 3.9]{SZ4} that the isomorphism
$\I_{1}\simeq \K(\ell^{2}(\N))\otimes (A\times_{\dot{\alpha}}^{\piso}\N)$, which we denote by $\Psi_{1}$ here, takes
each spanning element $[T_{n}(1-TT^{*})T_{y}^{*}]\otimes [v_{m}^{*}j_{A}(a)v_{x}]$ of the algebra
$\K(\ell^{2}(\N))\otimes (A\times_{\dot{\alpha}}^{\piso}\N)$ to the spanning element $\eta_{(m,n)}^{(x,y)}(a)$ of the
ideal $\I_{1}$ (see (\ref{span_I1})). Now, by a similar calculation to (\ref{eq3}), we have
\begin{eqnarray}\label{eq5}
\begin{array}{l}
(\dot{\rho}\times \dot{W})|_{\I_{1}}\bigg(\Psi_{1}\big([T_{n}(1-TT^{*})T_{y}^{*}]\otimes [v_{m}^{*}j_{A}(a)v_{x}]\big)\bigg)(e_{r}\otimes h)\\
=(\dot{\rho}\times \dot{W})|_{\I_{1}}\bigg(\eta_{(m,n)}^{(x,y)}(a)\bigg)(e_{r}\otimes h)\\
=e_{n}\otimes [\dot{\pi}(v_{m}^{*}j_{A}(a)v_{x})h]
\end{array}
\end{eqnarray}
if $r=y$, otherwise, zero for all $m,n,x,y,r\in \N$, $a\in A$, and $h\in H$. Thus, it follows by comparing (\ref{eq4}) and (\ref{eq5})
that we indeed have
\begin{align}
\label{Res_I1}
(\dot{\rho}\times \dot{W})|_{\I_{1}\simeq \K(\ell^{2}(\N))\otimes (A\times_{\dot{\alpha}}^{\piso}\N)}=\id\otimes \dot{\pi}.
\end{align}
Consequently, since the representations $\id$ and $\dot{\pi}$ are nonzero, it follows from (\ref{Res_I1}) that the restriction of
$\dot{\rho}\times \dot{W}$ to the ideal $\I_{1}$ must be nonzero.

Finally, since
$$(\dot{\rho}\times \dot{W})|_{\K(\ell^{2}(\N^{2}))\otimes A}=(\id\otimes \dot{\pi})|_{\K(\ell^{2}(\N^{2}))\otimes A}$$
and $\ker \dot{\pi}$ contains the algebra $\K(\ell^{2}(\N))\otimes A$ as an ideal (see (\ref{ext.seq.2}) and the definition of $\dot{\pi}$
in (\ref{pi.dot})), it follows that
$$\ker(\id\otimes \dot{\pi})=\K(\ell^{2}(\N))\otimes \ker \dot{\pi}\supset \K(\ell^{2}(\N))\otimes \big(\K(\ell^{2}(\N))\otimes A\big)
\simeq \K(\ell^{2}(\N^{2}))\otimes A,$$
and therefore, the representation $\dot{\rho}\times \dot{W}$ vanishes on the ideal $\K(\ell^{2}(\N^{2}))\otimes A$.
This completes the proof.
\end{proof}

\begin{remark}
\label{rmk4}
It therefore follows by Lemma \ref{from-alpha1} that, under the conditions of Theorem \ref{prim piso}, each primitive ideal of $A\times_{\alpha}^{\piso}\N^{2}$ coming from $\Prim (A\rtimes_{\dot{\alpha}}\Z)$ is the kernel of an irreducible representation
$\dot{\rho}\times \dot{W}$ corresponding to the pair $(\dot{\rho},\dot{W})$ induced by an element $[(\phi,z)]$ of
$\Delta(A)\times \TT/\sim^{(1)}\simeq \Prim (A\rtimes_{\dot{\alpha}}\Z)$. Let $\dot{\J}_{[(\phi,z)]}$ denote $\ker (\dot{\rho}\times \dot{W})$.
So, the map
$$\dot{\J}_{[(\phi,z)]}\mapsto \dot{\J}_{[(\phi,z)]}\cap \I_{1}=\ker ((\dot{\rho}\times \dot{W})|_{\I_{1}})
=\ker (\id\otimes \dot{\pi})=\K(\ell^{2}(\N))\otimes \ker \dot{\pi}$$
is a bijection between the subset of $\Prim (A\times_{\alpha}^{\piso}\N^{2})$ consisting of the primitive ideals $\dot{\J}_{[(\phi,z)]}$
and the closed subspace
$$F_{1}:=\big\{P\in\Prim (\I_{1}): \K(\ell^{2}(\N^{2}))\otimes A\subset P\big\}$$
of $\Prim (\I_{1})\simeq \Prim (A\times_{\dot{\alpha}}^{\piso}\N)$. Moreover, $F_{1}$ is homeomorphic to
$\Prim (A\rtimes_{\dot{\alpha}}\Z)$ by the composition of the following homeomorphisms
$$\Prim (A\rtimes_{\dot{\alpha}}\Z)\stackrel{}{\longrightarrow}
\big\{I\in \Prim (A\times_{\dot{\alpha}}^{\piso}\N): \K(\ell^{2}(\N))\otimes A\subset I\big\}
\stackrel{\textrm{the Rieffel homeomorphism}}{\longrightarrow} F_{1},$$
such that
$$[(\phi,z)]\mapsto \ker \big([\Ind_{\dot{\Z}_{\phi}}^{\Z}(\phi \times \gamma_{z}|_{\dot{\Z}_{\phi}})]\circ \dot{q}\big)
=\ker \dot{\pi}\mapsto \K(\ell^{2}(\N))\otimes \ker \dot{\pi}.$$
Therefore, the map $[(\phi,z)]\mapsto \dot{\J}_{[(\phi,z)]}$ embeds the set $\Prim (A\rtimes_{\dot{\alpha}}\Z)$ in
$\Prim (A\times_{\alpha}^{\piso}\N^{2})$ as a subset.
\end{remark}

Similarly, the semigroup dynamical system $(A,\N,\ddot{\alpha})$ gives rise to the following short exact sequence
\begin{align}
\label{ext.seq.3}
0 \longrightarrow \K(\ell^{2}(\N))\otimes A \stackrel{}{\longrightarrow} A\times_{\ddot{\alpha}}^{\piso}\N
\stackrel{\ddot{q}}{\longrightarrow} A\rtimes_{\ddot{\alpha}}\Z \longrightarrow 0.
\end{align}
Therefore, corresponding to each element $[(\phi,w)]$ of
$\Delta(A)\times \TT/\sim^{(2)}\simeq \Prim (A\rtimes_{\ddot{\alpha}}\Z)$, the composition
\begin{align}
\label{pi.ddot}
\big(\Ind_{\ddot{\Z}_{\phi}}^{\Z}(\phi \times \gamma_{w}|_{\ddot{\Z}_{\phi}})\big)\circ \ddot{q}
\end{align}
defines a nonzero irreducible representation $\ddot{\pi}$ of $(A\times_{\ddot{\alpha}}^{\piso}\N, k_{A},u)$ on a Hilbert space $H$, where $\ddot{\Z}_{\phi}$ denotes the stability group of $\phi$ when the group $\Z$ acts on $\Delta(A)$ via the action $\ddot{\alpha}$.
Hence, $\K(\ell^{2}(\N))\otimes (\ker \ddot{\pi})$ is a primitive ideal of the algebra
$\K(\ell^{2}(\N))\otimes (A\times_{\ddot{\alpha}}^{\piso}\N)$, which again by \cite[Corollary 3.12]{SZ4}, sits in $A\times_{\alpha}^{\piso}\N^{2}$
as an (essential) ideal $\I_{2}$ (note that, $\I_{2}$ is actually the ideal $\I_{\delta}$ in \cite{SZ4}). Now, we have

\begin{lemma}
\label{from-alpha2}
Define the maps
$$\ddot{\rho}:A\rightarrow B(\ell^{2}(\N)\otimes H)\ \ \textrm{and}\ \ \ddot{W}:\N^{2}\rightarrow B(\ell^{2}(\N)\otimes H)$$
by
$$(\ddot{\rho}(a)f)(m)=(\ddot{\pi}\circ k_{A})(\dot{\alpha}_{m}(a))f(m)\ \ \textrm{and}
\ \ \ddot{W}_{(m,n)}=T_{m}^{*}\otimes \overline{\ddot{\pi}}(u_{n})$$
respectively, for all $a\in A$, $f\in \ell^{2}(\N)\otimes H$, and $m,n\in\N$. Then, the pair $(\ddot{\rho},\ddot{W})$ is a covariant
partial-isometric representation of the system $(A,\N^{2}, \alpha)$ on the Hilbert space $\ell^{2}(\N)\otimes H$,
such that the corresponding (nondegenerate) representation $\ddot{\rho}\times \ddot{W}$ of $(A\times_{\alpha}^{\piso}\N^{2},i)$
is irreducible on $\ell^{2}(\N)\otimes H$, which lives on the ideal $\I_{2}\simeq \K(\ell^{2}(\N))\otimes (A\times_{\ddot{\alpha}}^{\piso}\N)$
but vanishes on $\K(\ell^{2}(\N^{2}))\otimes A$.
\end{lemma}

\begin{proof}
We skip the proof as it is similar to the proof of Lemma \ref{from-alpha1}.
\end{proof}

\begin{remark}
\label{rmk1}
Thus, by Lemma \ref{from-alpha2}, under the conditions of Theorem \ref{prim piso}, the primitive ideals of $A\times_{\alpha}^{\piso}\N^{2}$
derived from $\Prim (A\rtimes_{\ddot{\alpha}}\Z)$ are the kernels of the irreducible representations
$\ddot{\rho}\times \ddot{W}$ induced by elements $[(\phi,w)]$ of
$\Delta(A)\times \TT/\sim^{(2)}\simeq \Prim (A\rtimes_{\ddot{\alpha}}\Z)$. So, we denote these ideals by $\ddot{\J}_{[(\phi,w)]}$, and
hence, the map $[(\phi,w)]\in \Prim (A\rtimes_{\ddot{\alpha}}\Z)\mapsto \ddot{\J}_{[(\phi,w)]}\in \Prim (A\times_{\alpha}^{\piso}\N^{2})$
is an embedding (of sets).

In addition, as a refinement of Lemma \ref{QS2}, we would like to mention that the maps
$$P\mapsto\ind P, [(\phi,z)]\mapsto \dot{\J}_{[(\phi,z)]}, [(\phi,w)]\mapsto \ddot{\J}_{[(\phi,w)]},\ \textrm{and}\
[(\phi,(z,w))]\mapsto \J_{[(\phi,(z,w))]}$$
combine to give a bijective correspondence of the disjoint union (\ref{QS2}) onto $\Prim (A\times_{\alpha}^{\piso}\N^{2})$.
\end{remark}

\begin{prop}
\label{GCR piso}
Let $(A,\N^{2},\alpha)$ be a dynamical system consisting of a separable abelian $C^{*}$-algebra $A$ and an action $\alpha$ of
$\N^{2}$ on $A$ by automorphisms. Then $A\times_{\alpha}^{\piso}\N^{2}$ is GCR if and only if the orbit space
$\Z^{2}\backslash \Delta(A)$ is $T_{0}$.
\end{prop}

\begin{proof}
Recall that, by \cite[Theorem 5.6.2]{Murphy}, $A\times_{\alpha}^{\piso}\N^{2}$ is GCR if and only if
$$\ker q\ \ \ \ \textrm{and}\ \ \ \ A\rtimes_{\alpha}\Z^{2}\simeq C_{0}(\Delta(A))\rtimes_{\lt} \Z^{2}$$
are GCR (see (\ref{ext.seq.1})), and by \cite[Theorem 8.43]{W}, $A\rtimes_{\alpha}\Z^{2}$ is GCR if and only if the orbit space
$\Z^{2}\backslash \Delta(A)$ is $T_{0}$. So, it is enough to see that if $\Z^{2}\backslash \Delta(A)$ is a $T_{0}$ space, then the ideal
$\ker q$ is GCR. Suppose that $\Z^{2}\backslash \Delta(A)$ is $T_{0}$. To see that the algebra $\ker q$ is GCR, since, by
\cite[Corollary 3.13]{SZ4}, we have
$$\ker q/[\K(\ell^{2}(\N^{2}))\otimes A]\simeq
\big[\K(\ell^{2}(\N))\otimes (A\rtimes_{\dot{\alpha}}\Z)\big] \oplus \big[\K(\ell^{2}(\N))\otimes (A\rtimes_{\ddot{\alpha}}\Z)\big],$$
and $A$ is abelian, again, by \cite[Theorem 5.6.2]{Murphy}, it is enough show that the algebras
$$A\rtimes_{\dot{\alpha}}\Z\ \ \ \ \textrm{and}\ \ \ \ A\rtimes_{\ddot{\alpha}}\Z$$
are GCR. Let $\dot{\Z}\backslash \Delta(A)$ and $\ddot{\Z}\backslash \Delta(A)$ denote the orbit spaces corresponding
to the actions $\dot{\alpha}$ and $\ddot{\alpha}$ of $\Z$ on $A$, respectively. Suppose that
$\sigma:\Delta(A)\rightarrow \Z^{2}\backslash \Delta(A)$ and $\sigma_{1}:\Delta(A)\rightarrow \dot{\Z}\backslash \Delta(A)$ are the
orbit maps. One can see that the map $\Psi_{1}:\dot{\Z}\backslash \Delta(A)\rightarrow \Z^{2}\backslash \Delta(A)$ defined by
$$\dot{\Z}\cdot \phi\mapsto \Z^{2}\cdot \phi$$
is bijective, where $\dot{\Z}\cdot \phi$ denotes the $\Z$-orbit of $\phi\in \Delta(A)$ corresponding to the action $\dot{\alpha}$.
Moreover, we clearly have $\Psi_{1}\circ \sigma_{1}=\sigma$, by applying which, it follows the map $\Psi_{1}$ is actually
a homeomorphism. Therefore, the orbit space $\dot{\Z}\backslash \Delta(A)$ must also be $T_{0}$. A similar argument shows that
$\ddot{\Z}\backslash \Delta(A)$ is $T_{0}$, too, and hence, again by \cite[Theorem 8.43]{W}, the (group) crossed products
$A\rtimes_{\dot{\alpha}}\Z$ and $A\rtimes_{\ddot{\alpha}}\Z$ are GCR. This completes the proof.
\end{proof}

\begin{prop}
\label{CCR piso}
Let $(A,\N^{2},\alpha)$ be a dynamical system consisting of a separable abelian $C^{*}$-algebra $A$ and an action $\alpha$ of
$\N^{2}$ on $A$ by automorphisms. Then $A\times_{\alpha}^{\piso}\N^{2}$ is not CCR.
\end{prop}

\begin{proof}
Since $A\times_{\alpha}^{\piso}\N^{2}$ and
$$(B_{\Z^{2}}\otimes A)\rtimes_{\tau\otimes\alpha^{-1}}\Z^{2}\simeq C_{0}(X)\rtimes_{\lt} \Z^{2}$$
are Morita equivalent, it is enough to see that $C_{0}(X)\rtimes_{\lt} \Z^{2}$ is not CCR (see \cite[Proposition I.43]{W}).
Since for the element $((m,n),\phi)\in X$, where $m,n\in\Z$ and $\phi\in \Delta(A)$ (see (\ref{prod})), we have
$$\overline{\Z^{2}\cdot((m,n),\phi)}=\overline{\Z^{2}\times \{\phi\}}=\overline{\Z^{2}}\times \overline{\{\phi\}}
=\big(\overline{\Z}\times \overline{\Z}\big)\times \{\phi\}=\big[(\Z \cup \{\infty\})\times (\Z \cup \{\infty\})\big]\times \{\phi\},$$
it follows that the $\Z^{2}$-orbit of $((m,n),\phi)$ is not closed in $X$. Thus, by \cite[Theorem 8.44]{W},
$C_{0}(X)\rtimes_{\lt} \Z^{2}$ is not CCR.
\end{proof}

\section{The topology of $\Prim(A\times_{\alpha}^{\piso}\N^{2})$ when $A$ is separable and $\Z^{2}$ acts on $\Prim A$ freely}
\label{case 2}
Assume that in the system $(A,\N^{2},\alpha)$, the $C^*$-algebra $A$ is separable, and the action of $\Z^{2}$ on $\Prim A$ is free.
Now, consider the group dynamical system $(B_{\Z^{2}}\otimes A,\Z^{2},\tau\otimes\alpha^{-1})$ in which the algebra $(B_{\Z^{2}}\otimes A)$
is certainly separable and $\Z^{2}$ is an abelian (discrete) countable group. To describe the action of $\Z^{2}$ on
$\Prim(B_{\Z^{2}}\otimes A)$, first note that, by \cite[Theorem B.45]{RW}, $\Prim(B_{\Z^{2}}\otimes A)=\Prim(B_{\Z}\otimes B_{\Z} \otimes A)$ is homeomorphic to
\begin{align}
\label{prod4}
\Prim B_{\Z}\times \Prim B_{\Z}\times  \Prim A\simeq (\mathbb{Z} \cup \{\infty\})\times(\mathbb{Z} \cup \{\infty\})\times \Prim A.
\end{align}
Then, by a similar discussion to the one given at the beginning of \S\ref{case 1}, one can see that $\Z^{2}$ acts on
the product space (\ref{prod4}) as follows:
$$(m,n)\cdot ((r,s),P)=((r+m,s+n),P),\ \ \ \ (m,n)\cdot ((\infty,\infty),P)=((\infty,\infty),\alpha_{(-m,-n)}(P)),$$
$$(m,n)\cdot ((\infty,s),P)=((\infty,s+n),\alpha_{(-m,0)}(P))=((\infty,s+n),\dot{\alpha}_{-m}(P)),$$
and
$$(m,n)\cdot ((r,\infty),P)=((r+m,\infty),\alpha_{(0,-n)}(P))=((r+m,\infty),\ddot{\alpha}_{-n}(P))$$
for all $(m,n)\in \Z^{2}$, $P\in \Prim A$, and $r,s\in\Z$. So, it is not difficult to see that, in fact, $\Z^{2}$ also acts
on $\Prim(B_{\Z^{2}}\otimes A)$ freely, and therefore, by Theorem \ref{prim free action},
$$\Prim ((B_{\Z^{2}}\otimes A)\rtimes_{\tau\otimes\alpha^{-1}}\Z^{2})\simeq \Prim(A\times_{\alpha}^{\piso}\N^{2})$$
is homeomorphic to the quasi-orbit space
\begin{align}
\label{q-orbit}
\O\big(\Prim (B_{\Z^{2}}\otimes A)\big)=\O\big((\mathbb{Z} \cup \{\infty\})\times(\mathbb{Z} \cup \{\infty\})\times \Prim A\big).
\end{align}
But again, a similar discussion to the Lemma (\ref{QS-lem}) shows that the quasi-orbit space (\ref{q-orbit}), as a set, corresponds to
the disjoint union
\begin{align}
\label{q-orbit2}
\Prim A \sqcup \O_{1}(\Prim A)\sqcup \O_{2}(\Prim A)\sqcup \O(\Prim A),
\end{align}
where $\O_{1}(\Prim A)$, $\O_{2}(\Prim A)$, and $\O(\Prim A)$ are the quasi-orbit spaces homeomorphic to
$\Prim (A\rtimes_{\dot{\alpha}}\Z)$, $\Prim (A\rtimes_{\ddot{\alpha}}\Z)$, and $\Prim (A\rtimes_{\alpha}\Z^{2})$, respectively.
Now, we have:
\begin{theorem}
\label{prim piso2}
Let $(A,\N^{2},\alpha)$ be a dynamical system consisting of a separable $C^{*}$-algebra $A$ and an action $\alpha$ of
$\N^{2}$ on $A$ by automorphisms. Suppose that the action of $\Z^{2}$ on $\Prim A$ is free. Then, $\Prim (A\times_{\alpha}^{\piso}\N^{2})$
is homeomorphic to the disjoint union (\ref{q-orbit2}) equipped with the quotient topology in which the open sets are
in the following four forms:
\begin{itemize}
\item[(a)] $O\subset \Prim A$, where $O$ is open in $\Prim A$;
\item[(b)] $O\cup W_{1}$, where $O$ is a nonempty open subset of $\Prim A$ and $W_{1}$ is an open set in $\O_{1}(\Prim A)$;
\item[(c)] $O\cup W_{2}$, where $O$ is a nonempty open subset of $\Prim A$ and $W_{2}$ is an open set in $\O_{2}(\Prim A)$; and
\item[(d)] $O\cup W_{1}\cup W_{2}\cup W$, where $O$, $W_{1}$, and $W_{2}$ are nonempty open subsets of $\Prim A$,
$\O_{1}(\Prim A)$, and $\O_{2}(\Prim A)$, respectively, and $W$ is an open set in $\O(\Prim A)$.
\end{itemize}

\end{theorem}

\begin{proof}
We skip the proof as it follows by a similar argument to the proof of Theorem \ref{prim piso}, using the fact that, for any (group)
dynamical system $(B,G,\beta)$, the quasi-orbit map $\sigma:\Prim B\rightarrow\O(\Prim B)$ is continuous and open (see \cite[Lemma 6.12]{W}).
\end{proof}

\begin{remark}
\label{rmk 2}
Under the conditions of Theorem \ref{prim piso2}, the primitive ideals of $A\times_{\alpha}^{\piso}\N^{2}$ coming from
$\Prim (A\rtimes_{\alpha}\Z^{2})\simeq \O(\Prim A)$ are indeed the kernels of the irreducible representations
$(\Ind \pi)\circ q=(\tilde{\pi}\rtimes U)\circ q$ of $A\times_{\alpha}^{\piso}\N^{2}$, where $\pi$ is an irreducible representation
of $A$ such that $\ker\pi=P$ (see \S\ref{sec:pre}). Moreover, $(\Ind \pi)\circ q$ is actually  the associated representation
$\tilde{\pi}\times^{\piso} U$ of $A\times_{\alpha}^{\piso}\N^{2}$ corresponding to the covariant partial-isometric representation
$(\tilde{\pi},U)$ of $(A,\N^{2},\alpha)$. Therefore, each element of the closed subspace
$$F=\{\J\in\Prim (A\times_{\alpha}^{\piso}\N^{2}):\ker q\subset\J\}$$
of $\Prim (A\times_{\alpha}^{\piso}\N^{2})$ is the kernel of an irreducible representation $\tilde{\pi}\times^{\piso} U$ corresponding to the quasi-orbit $\O(P)$, which we denote by $\J_{\O(P)}$. Hence, the map $\O(P)\rightarrow\J_{\O(P)}$ is a homeomorphism of
$\O(\Prim A)\simeq \Prim (A\rtimes_{\alpha}\Z^{2})$ onto the closed subspace $F$.

Also, note that, the primitive ideals of $A\times_{\alpha}^{\piso}\N^{2}$
derived from $\Prim (A\rtimes_{\dot{\alpha}}\Z)\simeq\O_{1}(\Prim A)$ and $\Prim (A\rtimes_{\ddot{\alpha}}\Z)\simeq \O_{2}(\Prim A)$
can similarly be identified by looking at Lemmas \ref{from-alpha1} and \ref{from-alpha2}, respectively. To be more precise, for each
quasi-orbit $\O_{1}(P)\in\O_{1}(\Prim A)$ and $\O_{2}(P)\in\O_{2}(\Prim A)$, there is an irreducible representation $\pi$ of $A$
such that $P=\ker\pi$. Now, if $\Ind_{1} \pi$ and $\Ind_{2} \pi$  denote the induced representations of the crossed products
$A\rtimes_{\dot{\alpha}}\Z$ and $A\rtimes_{\ddot{\alpha}}\Z$, respectively, then the compositions
$$\dot{\pi}:=(\Ind_{1} \pi)\circ \dot{q}\ \ \textrm{and}\ \ \ddot{\pi}:=(\Ind_{2} \pi)\circ \ddot{q}$$ are nonzero irreducible
representations of the algebras $A\times_{\dot{\alpha}}^{\piso}\N$ and $A\times_{\ddot{\alpha}}^{\piso}\N$, respectively
(see (\ref{ext.seq.2}) and (\ref{ext.seq.3})). Then, the rest follows from Lemmas \ref{from-alpha1} and \ref{from-alpha2}. We denote the
primitive ideal of $A\times_{\alpha}^{\piso}\N^{2}$ corresponding to $\O_{1}(P)\in\O_{1}(\Prim A)$ by $\J_{\O_{1}(P)}$, and
similarly, the one corresponding to $\O_{2}(P)\in\O_{2}(\Prim A)$ by $\J_{\O_{2}(P)}$.
\end{remark}

\begin{remark}
\label{ZM result}
Recall that the primitive ideal space of any $C^{*}$-algebra is a locally compact space, and if a $C^{*}$-algebra is separable, then
its primitive ideal space is second countable. A (not necessarily Hausdorff) locally compact space $X$ is called \emph{almost Hausdorff}
if each locally compact subspace $V$ contains a relatively open nonempty Hausdorff subset (see \cite[Definition 6.1]{W}).
If a $C^{*}$-algebra is GCR, then its primitive ideal space is almost Hausdorrff with respect to the hull-kernel (Jacobson)
topology (see \cite[pages 171, 172]{W}). Therefore, If $A$ is a separable GCR $C^{*}$-algebra, then the spectrum $\widehat{A}$ of $A$
is a second countable almost Hausdorff locally compact space as it is homeomorphic to $\Prim A$.
Now, suppose that $(A,G,\alpha)$ is a group dynamical system in which the algebra $A$ is separable and $G$ is an abelian discrete countable
group. If $G$ acts on $\widehat{A}$ freely, then it follows from \cite{ZM} that
$A\rtimes_{\alpha}G$ is GCR if and only if $A$ is GCR and every $G$-orbit in $\widehat{A}$ is discrete. However, every $G$-orbit in
$\widehat{A}$ is discrete if and only if, for every $[\pi]\in\widehat{A}$, the map $G\rightarrow G\cdot[\pi]$ defined by
$$s\mapsto s\cdot[\pi]:=[\pi\circ\alpha_{s^{-1}}]$$
is a homeomorphism, which by \cite[Theorem 6.2 (Mackey-Glimm Dichotomy)]{W}, is equivalent to saying that the orbit space
$G\backslash\widehat{A}$ is $T_{0}$. Therefore, if (the abelian group) $G$ in the separable system $(A,G,\alpha)$ acts on
$\widehat{A}$ freely, then $A\rtimes_{\alpha}G$ is GCR if and only if $A$ is GCR and the orbit space $G\backslash\widehat{A}$ is $T_{0}$.
\end{remark}

\begin{prop}
\label{GCR piso2}
Let $(A,\N^{2},\alpha)$ be a dynamical system consisting of a separable $C^{*}$-algebra $A$ and an action $\alpha$ of
$\N^{2}$ on $A$ by automorphisms. Suppose that $\Z^{2}$ acts on $\widehat{A}$ freely. Then, $A\times_{\alpha}^{\piso}\N^{2}$ is GCR
if and only if $A$ GCR and the orbit space $\Z^{2}\backslash\widehat{A}$ is $T_{0}$.
\end{prop}

\begin{proof}
We skip the proof as it follows by a similar discussion to the proof of Proposition \ref{GCR piso} and Remark \ref{ZM result}.
\end{proof}

\section{Primitivity of $A\times_{\alpha}^{\piso}\N^{2}$}
\label{sec:last}
Finally, we have:
\begin{theorem}
\label{primitive}
Let $(A,\N^{2},\alpha)$ be a dynamical system consisting of a (nonzero) $C^{*}$-algebra $A$ and an action $\alpha$ of $\N^{2}$ on
$A$ by automorphisms. Then, $A\times_{\alpha}^{\piso}\N^{2}$ is primitive if and only $A$ is primitive.
\end{theorem}

\begin{proof}
If $A\times_{\alpha}^{\piso}\N^{2}$ is primitive, then it has a faithful nonzero irreducible representation $\rho$, and hence, $\ker \rho$,
which is the zero ideal, is a primitive ideal of $A\times_{\alpha}^{\piso}\N^{2}$. But, this ideal can only be derived from $\Prim A$ as all
primitive ideals of $A\times_{\alpha}^{\piso}\N^{2}$ except the ones derived from $\Prim A$ contain the algebra
$\K(\ell^{2}(\N^{2}))\otimes A$. Therefore, it follows from Proposition \ref{Ip} that the representation $\rho$ is the representation
$\tilde{\pi}\times V$ corresponding to a pair $(\tilde{\pi},V)$ induced by a (nonzero) irreducible representation $\pi$ of $A$.
It thus follows that the restriction of $\rho=\tilde{\pi}\times V$ to the ideal $\K(\ell^{2}(\N^{2}))\otimes A$ of $A\times_{\alpha}^{\piso}\N^{2}$, which is the representation $\id\otimes \pi$, is a (nonzero) faithful irreducible representation of $\K(\ell^{2}(\N^{2}))\otimes A$. So,
the algebra $\K(\ell^{2}(\N))\otimes A$ is primitive, which implies that $A$ must be primitive (in fact, $\ker \pi=\{0\}$).

Conversely, if $A$ is primitive, then it has a faithful nonzero irreducible representation $\pi$. Let $\tilde{\pi}\times V$ be the (nonzero)
irreducible representation of $A\times_{\alpha}^{\piso}\N^{2}$ corresponding to the pair $(\tilde{\pi},V)$ induced by the representation $\pi$ (see again Proposition \ref{Ip}). Now, the restriction $(\tilde{\pi}\times V)|_{\K(\ell^{2}(\N^{2}))\otimes A}=\id\otimes \pi$ is faithful as the
representations $\id$ and $\pi$ are (see \cite[Corollary B.11]{RW}). Therefore, since $\K(\ell^{2}(\N^{2}))\otimes A$ is an essential ideal of
$A\times_{\alpha}^{\piso}\N^{2}$ (see \cite[Corollary 3.13]{SZ4}), it follows that the representation $\tilde{\pi}\times V$ must be faithful. So,
the algebra $A\times_{\alpha}^{\piso}\N^{2}$ is primitive.
\end{proof}

\subsection*{Acknowledgements}
This work (Grant No. RGNS 64-102) was financially supported by Office of the Permanent Secretary, Ministry of Higher Education, Science,
Research and Innovation.

\end{document}